\documentclass[12pt]{article}

\usepackage[T2A]{fontenc}
\usepackage[utf8]{inputenc}
\usepackage[english]{babel}

\usepackage{amsmath, amssymb, amsthm}
\usepackage{bm} 

\usepackage{graphicx}
\usepackage{url}

\newcommand{\Li}{\mathrm{Li}_2} 

\newtheorem{theorem}{Theorem}[section]

\newtheorem{definition}[theorem]{Definition}
\newtheorem{example}[theorem]{Example}

\begin{document}

\title{A Family of Generating Functions for Reciprocal Binomial Coefficients and Its Applications}

\author{ Dmitry Kruchinin, Vladimir Kruchinin}

\maketitle

\begin{abstract}
A generating function for reciprocal binomial coefficients is written down, integral representations of this function are obtained, generating functions for sums of reciprocal binomial coefficients are derived, new identities are obtained, including identities connecting reciprocal binomial coefficients with harmonic numbers and Fibonacci numbers. The application of the found functions for evaluating infinite numerical sequences involving reciprocal binomial coefficients is demonstrated.
\end{abstract}


\section{Introduction}

A reciprocal binomial coefficient is the expression
$$
T(n,m)={\binom{n}{m}}^{-1}
$$
$T(n,m)$ represents the following triangle:
$$
\begin{array}{ccccc}\\
\frac{1}{\binom{0}{0}}\\
\frac{1}{\binom{1}{0}}&\frac{1}{\binom{1}{1}}\\
\frac{1}{\binom{2}{0}}&\frac{1}{\binom{2}{1}}&\frac{1}{\binom{2}{2}}\\
\vdots& \vdots & \vdots &  \ddots\\
\frac{1}{\binom{n}{0}}&\frac{1}{\binom{n}{1}}&\frac{1}{\binom{n}{2}}&\ldots&\frac{1}{\binom{n}{n}}\\

\end{array}
$$

The combinatorial interpretation of this triangle is based on an exponential representation, similar to problems from \cite{Graham}:
$$
T_e(n,m)=\frac{n!}{\binom{n}{m}}=m!\,(n-m)!
$$
This triangle is known in the On-Line Encyclopedia of Integer Sequences under the number A098361\cite{A098361}. The simplest combinatorial interpretation: the number of pairs of permutations of elements from two sets consisting of $k$ and $(n-k)$ elements. In quantum physics, for describing a quantum state \cite{RBS}. There is also an interpretation related to: This sequence gives the variance of the 2-dimensional Polynomial Chaoses \cite{Ghanem}. There is also a combinatorial interpretation for the row sums of the triangle $T_e(n,m)$ (see sequence A003149 \cite{A003149}).

There is a number of studies devoted to various sums with reciprocal binomial coefficients \cite{Sury1,Pla,Trif,Mansour,Sury2,Ji,Belbachir,Batir22}, in which a number of identities, explicit and recurrent formulas, and their generating functions are obtained. There are also works devoted to the study of reciprocal central binomial coefficients, providing their generating functions and explicit formulas \cite{Sprugnoli,Sofo,Batir22,Bhat2024}. There are also identities for the product of reciprocal binomial coefficients and harmonic numbers \cite{Sofo,Boyadzhiev,Kalmykov}. However, a generating function for the triangle of reciprocal binomial coefficients was absent.
The goal of this research is to obtain a generating function for the triangle of reciprocal coefficients
$$
A(x,y)=\sum_{n\geqslant 0}\sum_{m\geqslant 0} {\binom{n}{m}}^{-1}x^n\,y^m
$$
and its applications for obtaining new identities and generating functions for sums of reciprocal binomial coefficients.

\section{Preliminaries}

Classical methods for working with generating functions are described in detail in fundamental works \cite{Wilf,Stanley,Flajolet}.

\subsection{Notation}
Let us introduce notation used in this work:
\begin{itemize}
    \item $[x^n\,y^m]G(x,y)$ — coefficient of $x^n y^m$ in the expansion of the generating function $G(x,y)$;
    \item $\delta_{n,k}$ — Kronecker delta ($\delta_{n,k}=1$ when $n=k$, otherwise $0$).
\end{itemize}

\subsection{Standard Generating Functions (Reference)}
For convenience, we provide the explicit form of some classical generating functions:
\[
\begin{aligned}
    \log(1+x) &= \sum_{n>0} \frac{(-1)^{n-1}}{n}\,x^n,\\
    \mathrm{Li}_2(x) &= \sum_{n>0} \frac{x^n}{n^2},\\
    H(x) &= \frac{1}{1-x}\log\left(\frac{1}{1-x}\right) = \sum_{n \ge 0} H_n x^n,\\
    F(x) &= \frac{x}{1-x-x^2} = \sum_{n \ge 0} F_n x^n,
\end{aligned}
\]
where $H_n$ is the $n$-th harmonic number, and $F_n$ is the $n$-th Fibonacci number ($F_0=0, F_1=1$).

\subsection{Riordan Arrays and the Main Theorem}

Let formal power series be given:
\[
F(x)=\sum_{n\geqslant 0} f_n x^n, \qquad G(x)=\sum_{n>0} g_n x^n \quad (G(0)=0).
\]
\begin{definition}
    A Riordan array generated by the pair $(F(x), G(x))$ is the family of numbers
    \[
    R(n,k)=[x^n]\,F(x)\,G(x)^k, \qquad n,k\ge 0.
    \]
\end{definition}

For Riordan arrays, the following important theorem holds, which will be used in proving the main results.

\begin{theorem}[Composition Theorem for Riordan Arrays \cite{Shapiro}]
    Let $Q(x)=\sum_{k\ge 0} q_k x^k$ be an arbitrary formal power series. Then the coefficients of the generating function
    \[
    B(x)=F(x)\,Q\bigl(G(x)\bigr)
    \]
    are expressed via the Riordan array $R(n,k)$:
    \begin{equation}\label{eq:riordan-composition}
    [x^n]B(x)=\sum_{k=0}^n R(n,k)\,q_k.
    \end{equation}
\end{theorem}

\subsection{Compositions of Bivariate Generating Functions}
The work will also require some results on composition of two-dimensional generating functions.\! The main idea is as follows.\! If  $G(x,y)=\sum_{n,m} g_{n,m}x^n y^m$ and $G(0,0)=0$, then the coefficients of the composition $F\bigl(G(x,y)\bigr)$ are expressed via powers of $G(x,y)$. Namely, if $G^{\Delta}(n,m,k)=[x^n y^m]G(x,y)^k$, then
\begin{equation}\label{eq:2d-composition-idea}
[x^n y^m]F\bigl(G(x,y)\bigr)=\sum_{k\ge 0} G^{\Delta}(n,m,k)\,f_k,
\end{equation}
where $F(z)=\sum_{k\ge 0} f_k z^k$. A detailed exposition of this method and explicit formulas for $G^{\Delta}(n,m,k)$ can be found in \cite{KruBivar}.

Next we will need two simple special cases, easily obtained from the general scheme.

\subsubsection{Diagonal Sum}
If $F(x,y)=\sum_{n,m} f(n,m)x^n y^m$, then upon substitution $y=x$, we get
\begin{equation}\label{eq:diag-sum}
[x^n]F(x,x)=\sum_{m=0}^n f(n-m,m).
\end{equation}

\subsubsection{Weighted Row Sum}
Similarly, upon substitution $y=a$ (where $a$ is a constant), we have
\begin{equation}\label{eq:row-sum}
[x^n]F(x,a)=\sum_{m=0}^n f(n,m)\,a^m.
\end{equation}

These two relations will be actively used in Section 4 to obtain generating functions for various sums of reciprocal binomial coefficients.

\section{A Family of Generating Functions for Reciprocal Binomial Coefficients}

\subsection{Generating Function for Reciprocal Binomial Coefficients $\binom{n}{m}^{-1}$}

To write the generating function for reciprocal binomial coefficients $A(x,y)$, we prove the following theorem.

\begin{theorem}
The function
\begin{equation}\label{Lgf77}
A(x,y)=-{{y\,\log \left(\left(1-x\right)\,\left(1-x\,y\right)\right)
 }\over{\left(-x\,y+y+1\right)^2}}+{{x\,y^2}\over{\left(1-x\,y\right)
 \,\left(-x\,y+y+1\right)}}+{{1}\over{1-x}}
\end{equation}
is the generating function for the set of reciprocal binomial coefficients
$$
[x^n\,y^m]A(x,y)=\binom{n}{m}^{-1}.
$$
\end{theorem}

\begin{proof}
We use the known identity \cite{Prudnikov}
\begin{equation}\label{LidenPrudnik}
\sum_{k=0}^{n}{{(-1)^{k}\,{{n}\choose{k}}}\frac{m}{m
 +k}}=\binom{n+m}{n}^{-1}
\end{equation}

The left-hand side of this identity can be considered as a composition of the Riordan array
$$
\left(\frac{1}{y-1},\frac{y}{y-1}\right)
$$
and the generating function
\[
F(x,y) = \sum_{n=0}^{\infty} \sum_{k=1}^{\infty} \frac{k}{n+k} \, x^n y^k,
\]
Let's find the generating function $F(x,y)$.
Note that
   \[
   \frac{1}{n+k} = \int_0^1 t^{\,n+k-1} \, dt, \quad n \ge 0,\; k \ge 1.
   \]
   Hence,
   \[
   \frac{k}{n+k} = k \int_0^1 t^{\,n+k-1} \, dt.
   \]

Substitute into the double sum
   \[
   F(x,y) = \sum_{n=0}^{\infty} \sum_{k=1}^{\infty} k \, x^n y^k \int_0^1 t^{\,n+k-1} \, dt.
   \]
   In the domain $|x|<1$, $|y|<1$, we can interchange the order of summation and integration:
   \[
   F(x,y) = \int_0^1 \left( \sum_{n=0}^{\infty} (xt)^n \right) 
           \left( \sum_{k=1}^{\infty} k \, y^k t^{\,k-1} \right) dt.
   \]

 Compute the series:
   \[
   \sum_{n=0}^{\infty} (xt)^n = \frac{1}{1-xt},
   \]
   \[
   \sum_{k=1}^{\infty} k \, y^k t^{\,k-1} 
       = \frac{1}{t} \sum_{k=1}^{\infty} k (yt)^k 
       = \frac{1}{t} \cdot \frac{yt}{(1-yt)^2}
       = \frac{y}{(1-yt)^2}.
   \]
   Substituting:
   \[
   F(x,y) = y \int_0^1 \frac{1}{(1-xt)(1-yt)^2} \, dt.
   \]

Perform partial fraction decomposition.   Find coefficients $A$, $B$, $C$ such that
   \[
   \frac{1}{(1-xt)(1-yt)^2} = \frac{A}{1-xt} + \frac{B}{1-yt} + \frac{C}{(1-yt)^2}.
   \]
   Multiplying by $(1-xt)(1-yt)^2$, we get the identity
   \[
   1 = A(1-yt)^2 + B(1-xt)(1-yt) + C(1-xt).
   \]
   By substitution we find:
   \[
   \begin{aligned}
   t = \frac{1}{x} &\Rightarrow A = \frac{x^2}{(x-y)^2},\\[2mm]
   t = \frac{1}{y} &\Rightarrow C = \frac{y}{y-x} = -\frac{y}{x-y},\\[2mm]
   t = 0 &\Rightarrow B = 1 - A - C = \frac{-xy}{(x-y)^2}.
   \end{aligned}
   \]

Perform integration; denote
   \[
   J(x,y) = \int_0^1 \frac{dt}{(1-xt)(1-yt)^2}.
   \]
   Then
   \[
   \begin{aligned}
   J(x,y) &= \int_0^1 \left[ \frac{A}{1-xt} + \frac{B}{1-yt} + \frac{C}{(1-yt)^2} \right] dt \\[1mm]
          &= -\frac{A}{x} \ln(1-x) - \frac{B}{y} \ln(1-y) + \frac{C}{1-y}.
   \end{aligned}
   \]
   Substitute $A$, $B$, $C$:
   \[
   \begin{aligned}
   -\frac{A}{x} &= -\frac{x}{(x-y)^2},\qquad
   -\frac{B}{y} = \frac{x}{(x-y)^2},\qquad
   \frac{C}{1-y} = -\frac{y}{(x-y)(1-y)}.
   \end{aligned}
   \]
   Hence,
   \[
   J(x,y) = \frac{x}{(x-y)^2} \ln\!\left( \frac{1-y}{1-x} \right) - \frac{y}{(x-y)(1-y)}.
   \]

Obtain the final expression for $F(x,y)$. Multiply $J(x,y)$ by $y$:
   \[
   F(x,y) = y \, J(x,y) 
          = \frac{xy}{(x-y)^2} \ln\!\left( \frac{1-y}{1-x} \right) 
            - \frac{y^2}{(x-y)(1-y)}.
   \]

Check correctness: the obtained formula is valid for $|x|<1$, $|y|<1$, $x \neq y$.
   \begin{itemize}
     \item For $x=0$:
       \[
       F(0,y) = \frac{y}{1-y} = \sum_{k=1}^{\infty} y^k,
       \]
       which agrees with the original series for $n=0$ (since $\frac{k}{0+k}=1$).
     \item The limit as $x \to y$ can be computed, yielding the one-dimensional case.
   \end{itemize}

Now note that the right-hand side of identity \ref{LidenPrudnik} equals 1 at zero. Therefore, for the composition we write the function $1+F(x,y)$ and $\binom{n}{m}$ - triangle ($m\leqslant n$), then the desired generating function is
$$
A(x,y)=\frac{1}{xy-1}\left(1+F\left(x,\frac{xy}{xy-1}\right)\right)
$$ 
After substitution and simplification we obtain the desired function
$$A(x,y)=-{{y\,\log \left(\left(1-x\right)\,\left(1-x\,y\right)\right)
 }\over{\left(-x\,y+y+1\right)^2}}+{{x\,y^2}\over{\left(1-x\,y\right)
 \,\left(-x\,y+y+1\right)}}+{{1}\over{1-x}}$$ 

   Thus, the theorem is proved.
\end{proof}

\subsection{Generating Functions for Reciprocal Binomial Coefficients
$\binom{2n}{m}^{-1}$, $\binom{n}{2m}^{-1}$, $\binom{2n-1}{m}^{-1}$, $\binom{n}{2m-1}^{-1}$}

Using the main generating function $A(x,y)$, one can obtain generating functions for subsets of reciprocal binomial coefficients in which either rows or columns have a specific parity. These functions are obtained using standard methods for extracting even and odd parts of a power series.

\subsubsection{Extracting Even and Odd Rows}

To extract rows with even indices in the triangle $\binom{n}{m}^{-1}$, the following relation is used:

\[
\frac{A(\sqrt{x}, y) + A(-\sqrt{x}, y)}{2} = \sum_{n \geq 0} \sum_{m \geq 0} \binom{2n}{m}^{-1} x^n y^m.
\]

Denote this function by $A_1(x,y)$:

\begin{equation}\label{A1}
A_1(x,y) = \frac{A(\sqrt{x}, y) + A(-\sqrt{x}, y)}{2}, \quad
[x^n y^m] A_1(x,y) = \binom{2n}{m}^{-1}.
\end{equation}

To extract rows with odd indices, a similar formula is applied:

\[
\frac{A(\sqrt{x}, y) - A(-\sqrt{x}, y)}{2\sqrt{x}} = \sum_{n \geq 1} \sum_{m \geq 0} \binom{2n-1}{m}^{-1} x^n y^m.
\]

Denote this function by $A_3(x,y)$:

\begin{equation}\label{A3}
A_3(x,y) = \frac{A(\sqrt{x}, y) - A(-\sqrt{x}, y)}{2\sqrt{x}}, \quad
[x^n y^m] A_3(x,y) = \binom{2n-1}{m}^{-1}.
\end{equation}

\subsubsection{Extracting Even and Odd Columns}

Similarly, to extract columns with even indices, we use:

\[
\frac{A(x, \sqrt{y}) + A(x, -\sqrt{y})}{2} = \sum_{n \geq 0} \sum_{m \geq 0} \binom{n}{2m}^{-1} x^n y^m.
\]

Denote this function by $A_2(x,y)$:

\begin{equation}\label{A2}
A_2(x,y) = \frac{A(x, \sqrt{y}) + A(x, -\sqrt{y})}{2}, \quad
[x^n y^m] A_2(x,y) = \binom{n}{2m}^{-1}.
\end{equation}

To extract columns with odd indices, we apply:

\[
\frac{A(x, \sqrt{y}) - A(x, -\sqrt{y})}{2\sqrt{y}} = \sum_{n \geq 0} \sum_{m \geq 1} \binom{n}{2m-1}^{-1} x^n y^m.
\]

Denote this function by $A_4(x,y)$:

\begin{equation}\label{A4}
A_4(x,y) = \frac{A(x, \sqrt{y}) - A(x, -\sqrt{y})}{2\sqrt{y}}, \quad
[x^n y^m] A_4(x,y) = \binom{n}{2m-1}^{-1}.
\end{equation}

\subsubsection{Explicit Expressions for the Obtained Functions}

Substituting the expression for $A(x,y)$ from (\ref{Lgf77}) into formulas (\ref{A1})–(\ref{A4}), explicit analytic expressions for $A_1$, $A_2$, $A_3$, $A_4$ can be obtained. They represent rational-logarithmic functions similar to the original $A(x,y)$, but with a more complex structure due to radicals.

For example, for $A_1(x,y)$ we obtain:

\[
A_1(x,y) =\! \left[ 
\frac{y \log\bigl((1\!-\sqrt{x})(1\!-\sqrt{x}\,y)\bigr)}{(-1 - y+ \sqrt{x}\,y)^2}\!
+\! \frac{\sqrt{x}\,y^2}{(1\!-\sqrt{x}\,y)(1\! + y\! - \sqrt{x}\,y)}
+ \frac{1}{1\!-\sqrt{x}} \right.
\]
\[
\left. 
- \frac{y \log\bigl((1+\sqrt{x})(1+\sqrt{x}\,y)\bigr)}{(1 + y + \sqrt{x}\,y)^2}
+ \frac{(-\sqrt{x})\,y^2}{(1+\sqrt{x}\,y)(1 + y + \sqrt{x}\,y)}
+ \frac{1}{1+\sqrt{x}} \right]\frac{1}{2}. 
\]

where the symbol $(\sqrt{x} \to -\sqrt{x})$ means repeating the previous terms with $\sqrt{x}$ replaced by $-\sqrt{x}$.

Similar expressions can be written for $A_2$, $A_3$, $A_4$, but they are quite bulky. However, these representations are useful for further analysis, for example, when computing sums or searching for identities.

\subsubsection{Combinatorial Interpretation}

The obtained functions correspond to the following combinatorial objects:
\begin{itemize}
\item $A_1(x,y)$ — reciprocal binomial coefficients for even rows of Pascal's triangle,
\item $A_2(x,y)$ — reciprocal binomial coefficients for even columns,
\item $A_3(x,y)$ — reciprocal binomial coefficients for odd rows,
\item $A_4(x,y)$ — reciprocal binomial coefficients for odd columns.
\end{itemize}

These functions can be used to study sums and identities related to the corresponding subsets of reciprocal binomial coefficients, as will be demonstrated in subsequent sections.

\section{Integral Representation of the Generating\\ Function \(A(x,y)\)}\label{sec:integral-repr}

Integral representations of the generating function \(A(x,y)\) allow us to obtain new families of generating functions for reciprocal binomial coefficients with additional factors of the form \(1/n\) or \(1/m\). Consider three main integral transformations:

\[
I(x,y)=\int A(x,y)\,dx, \quad
J(x,y)=\iint A(x,y)\,dx\,dy, \quad
\]
\[
K(x,y)=\frac{\partial}{\partial x} J(x,y) = \int A(x,y)\,dy.
\]

Each of these functions generates its own triangle of reciprocal binomial coefficients, which can be further modified by extracting even/odd rows or columns.

\subsection{Generating Function \(I(x,y)\)}

The function \(I(x,y)\) is obtained by integrating \(A(x,y)\) with respect to \(x\):

\[
I(x,y) = \int A(x,y)\,dx
       = \frac{\log\bigl(1 - x + (x^2 - x)y\bigr)}{(x-1)y - 1} + C.
\]

The constant of integration \(C = 0\), since at \(x=0\) the series should start with zero coefficients. Then the expansion coefficients of \(I(x,y)\) are:

\[
[x^n y^m] I(x,y) = \frac{1}{n \binom{n-1}{m}}, \quad n \geq 1,\; 0 \leq m \leq n-1.
\]

\subsubsection{Modifications of \(I(x,y)\) by Extracting Even/Odd Parts}

Applying standard formulas for extracting even and odd components, we obtain four new generating functions:

\begin{align*}
I_1(x,y) &= \frac{I(\sqrt{x}, y) + I(-\sqrt{x}, y)}{2}, &
[x^n y^m] I_1(x,y) &= \frac{1}{2n \binom{2n-1}{m}},\\
I_2(x,y) &= \frac{I(x, \sqrt{y}) + I(x, -\sqrt{y})}{2}, &
[x^n y^m] I_2(x,y) &= \frac{1}{n \binom{n-1}{2m}},\\
I_3(x,y) &= \frac{I(\sqrt{x}, y) - I(-\sqrt{x}, y)}{2\sqrt{x}}, &
[x^n y^m] I_3(x,y) &= \frac{1}{(2n-1) \binom{2n-2}{m}},\\
I_4(x,y) &= \frac{I(x, \sqrt{y}) - I(x, -\sqrt{y})}{2\sqrt{y}}, &
[x^n y^m] I_4(x,y) &= \frac{1}{n \binom{n-1}{2m-1}}.
\end{align*}

\subsection{Generating Function \(J(x,y)\)}

The function \(J(x,y)\) is obtained by double integration of \(A(x,y)\):

\[
J(x,y) = \iint A(x,y)\,dx\,dy
       = \frac{\Li\bigl(x(1 + (1-x)y)\bigr) - \Li(x)}{1-x},
\]
where \(\Li(z) = \sum_{n \geq 1} \frac{z^n}{n^2}\) is the dilogarithm.

The expansion coefficients of \(J(x,y)\) are:

\[
[x^n y^m] J(x,y) = \frac{1}{n \cdot m \cdot \binom{n-1}{m-1}}, \quad n \geq 1,\; 1 \leq m \leq n.
\]

\subsubsection{Modifications of \(J(x,y)\) by Extracting Even/Odd Parts}

Similarly to the previous case, we obtain:

\begin{align*}
J_1(x,y) &= \frac{J(\sqrt{x}, y) + J(-\sqrt{x}, y)}{2}, &
[x^n y^m] J_1(x,y) &= \frac{1}{2n \cdot m \cdot \binom{2n-1}{m-1}},\\
J_2(x,y) &= \frac{J(x, \sqrt{y}) + J(x, -\sqrt{y})}{2}, &
[x^n y^m] J_2(x,y) &= \frac{1}{n \cdot 2m \cdot \binom{n-1}{2m-1}},\\
J_3(x,y) &= \frac{J(\sqrt{x}, y) - J(-\sqrt{x}, y)}{2\sqrt{x}}, &
[x^n y^m] J_3(x,y) &= \frac{1}{(2n-1) \cdot m \cdot \binom{2n-2}{m-1}},\\
J_4(x,y) &= \frac{J(x, \sqrt{y}) - J(x, -\sqrt{y})}{2\sqrt{y}}, &
[x^n y^m] J_4(x,y) &= \frac{1}{n \cdot (2m-1) \cdot \binom{n-1}{2m-2}}.
\end{align*}

\subsection{Generating Function \(K(x,y)\)}

The function \(K(x,y)$ is the derivative of $J(x,y)$ with respect to $x$:

\[
K(x,y) = \frac{\partial J(x,y)}{\partial x}
       =
\]
\[
 = \frac{1}{1-x}\left[ \frac{\log(1-x)}{x} - \frac{(1 + (1-2x)y) \log\bigl(1 - x(1 + (1-x)y)\bigr)}{x(1 + (1-x)y)} \right]+
\]
\[
+ \frac{\mathrm{Li}_2\bigl(x(1 + (1-x)y)\bigr) - \mathrm{Li}_2(x)}{(1-x)^2}.
\]
The expansion coefficients of \(K(x,y)\) are:

\[
[x^n y^m] K(x,y) = \frac{1}{m \cdot \binom{n}{m-1}}, \quad n \geq 0,\; 1 \leq m \leq n+1.
\]

\subsubsection{Modifications of \(K(x,y)\) by Extracting Even/Odd Parts}

\begin{align*}
K_1(x,y) &= \frac{K(\sqrt{x}, y) + K(-\sqrt{x}, y)}{2}, &
[x^n y^m] K_1(x,y) &= \frac{1}{m \cdot \binom{2n}{m-1}},\\
K_2(x,y) &= \frac{K(x, \sqrt{y}) + K(x, -\sqrt{y})}{2}, &
[x^n y^m] K_2(x,y) &= \frac{1}{2m \cdot \binom{n}{2m-1}},\\
K_3(x,y) &= \frac{K(\sqrt{x}, y) - K(-\sqrt{x}, y)}{2\sqrt{x}}, &
[x^n y^m] K_3(x,y) &= \frac{1}{m \cdot \binom{2n-1}{m-1}},\\
K_4(x,y) &= \frac{K(x, \sqrt{y}) - K(x, -\sqrt{y})}{2\sqrt{y}}, &
[x^n y^m] K_4(x,y) &= \frac{1}{(2m-1) \cdot \binom{n}{2m-2}}.
\end{align*}

\subsection{Summary Table of the Family of Generating Functions}

Table \ref{tabl1} presents the complete family of generating functions obtained from \(A(x,y)\) by integration and extraction of even/odd components.

\begin{table}[h]
\centering
\caption{Family of generating functions for reciprocal binomial coefficients}
\label{tabl1}
\begin{tabular}{|c|c|c|c|c|}
\hline
 & \(A(x,y)\) & \(I(x,y)\) & \(J(x,y)\) & \(K(x,y)\) \\
\hline
0 & \(\dfrac{1}{\binom{n}{m}}\) & \(\dfrac{1}{n\binom{n-1}{m}}\) & \(\dfrac{1}{n m \binom{n-1}{m-1}}\) & \(\dfrac{1}{m \binom{n}{m-1}}\) \\
\hline
1 & \(\dfrac{1}{\binom{2n}{m}}\) & \(\dfrac{1}{2n \binom{2n-1}{m}}\) & \(\dfrac{1}{2n m \binom{2n-1}{m-1}}\) & \(\dfrac{1}{m \binom{2n}{m-1}}\) \\
\hline
2 & \(\dfrac{1}{\binom{n}{2m}}\) & \(\dfrac{1}{n \binom{n-1}{2m}}\) & \(\dfrac{1}{n \cdot 2m \cdot \binom{n-1}{2m-1}}\) & \(\dfrac{1}{2m \binom{n}{2m-1}}\) \\
\hline
3 & \(\dfrac{1}{\binom{2n-1}{m}}\) & \(\dfrac{1}{(2n-1) \binom{2n-2}{m}}\) & \(\dfrac{1}{(2n-1) m \binom{2n-2}{m-1}}\) & \(\dfrac{1}{m \binom{2n-1}{m-1}}\) \\
\hline
4 & \(\dfrac{1}{\binom{n}{2m-1}}\) & \(\dfrac{1}{n \binom{n-1}{2m-1}}\) & \(\dfrac{1}{n (2m-1) \binom{n-1}{2m-2}}\) & \(\dfrac{1}{(2m-1) \binom{n}{2m-2}}\) \\
\hline
\end{tabular}
\end{table}

\subsection{Combinatorial Interpretation}

Each row of the table corresponds to a specific subset of reciprocal binomial coefficients:
\begin{itemize}
\item Row 0 — the original triangle \(\binom{n}{m}^{-1}\).
\item Rows 1 and 3 — even and odd rows of the triangle.
\item Rows 2 and 4 — even and odd columns of the triangle.
\end{itemize}
Integral transformations add factors \(n\) or \(m\) in the denominator, which corresponds to taking partial sums or weighting the coefficients.

Thus, the constructed family of generating functions provides a unified framework for studying a wide class of sums and identities related to reciprocal binomial coefficients.

\section{Convergence Analysis of the Constructed\\ Generating Functions}\label{sec:convergence}

All generating functions obtained in this work are initially considered as elements of the ring of formal power series $\mathbb{C}[[x,y]]$.
The corresponding identities are established by methods of formal transformations
(composition of Riordan arrays, integration and differentiation of formal series).
The obtained analytical expressions (containing logarithms, dilogarithms,
rational functions) define the \emph{analytic continuation} of the formal
series in domains of complex variables. This section briefly describes the domains
of convergence of the corresponding power series, justifying the legitimacy
of substituting specific values and limit transitions used in Section~\ref{sec:infinite-sums}.

\subsection{Convergence of the Main Series $A(x,y)$}

Consider the two-dimensional power series
\[
A(x,y) = \sum_{n\geq 0}\sum_{m=0}^n \binom{n}{m}^{-1} x^n y^m .
\]
To estimate its radius of convergence, we use the asymptotics of the coefficients.
For fixed $\alpha = m/n \in (0,1)$ the following holds:
\[
\binom{n}{\alpha n}^{-1} \;\sim\;
\frac{\sqrt{2\pi n\alpha(1-\alpha)}}{\alpha^{\alpha n}(1-\alpha)^{(1-\alpha)n}},
\qquad n\to\infty .
\]
The function $\rho(\alpha) = \alpha^{\alpha}(1-\alpha)^{1-\alpha}$ reaches its minimum
$\rho_{\min} = 1/2$ at $\alpha=1/2$ and its maximum value $1$ at the ends of the interval
$\alpha=0,1$. Taking into account the factor $y^m = y^{\alpha n}$ leads to the convergence condition
\[
|x|\cdot \max_{0\leq\alpha\leq 1} \left( \frac{|y|^\alpha}{\rho(\alpha)} \right) < 1 .
\]
This inequality holds, for example, when
\begin{equation}\label{eq:conv-cond-asympt}
|x| < 1, \qquad |xy| < 1, \qquad |x|\sqrt{|y|} < 2 .
\end{equation}

The analytical expression (\ref{Lgf77}) for $A(x,y)$ has singularities when
denominators vanish, which determines a broader domain of analyticity:
\begin{equation}\label{eq:conv-cond-analytic}
|x|<1, \qquad |xy|<1, \qquad |1+y-xy|>0 .
\end{equation}
In particular, this ensures convergence for:
\begin{itemize}
\item $y=1$: $|x|<1$;
\item $x=1$: $|y|<1$;
\item $x=y$: $|x|<(\sqrt{5}-1)/2$ (the smaller root of the equation $1-x-x^2=0$).
\end{itemize}

\subsection{Convergence of Integral Modifications}

The functions $I(x,y)$, $J(x,y)$, $K(x,y)$ are obtained by formal integration
and differentiation of the series $A(x,y)$. Since these operations do not reduce
the radius of convergence of a formal series, the mentioned functions are analytic, at least,
in the same domain (\ref{eq:conv-cond-analytic}) as the original series.
Moreover, explicit expressions for these functions (Section~\ref{sec:integral-repr})
allow us to refine their domains:
\begin{itemize}
\item $I(x,y)$ is analytic for $|x|<1$, $|xy|<1$ and away from zeros of the denominator $(x-1)y-1$;
\item $J(x,y)$ is analytic where $\mathrm{Li}_2(x)$ and
      $\mathrm{Li}_2\!\bigl(x(1+(1-x)y)\bigr)$ are defined;
\item $K(x,y)$ is analytic in the domain where the series for $J(x,y)$ and their derivatives converge.
\end{itemize}

\subsection{Convergence of Modifications with Parity}\label{sec:parity-mods}

The functions $A_1, A_2, A_3, A_4$ (Section~\ref{sec:parity-mods}) and their integral
analogues are obtained by substitutions $\sqrt{x}, -\sqrt{x}, \sqrt{y}, -\sqrt{y}$
into the corresponding series. These substitutions transform the convergence domain
(\ref{eq:conv-cond-analytic}) into domains of the form:
\[
|\sqrt{x}|<1, \qquad |\sqrt{x}y|<1, \qquad |1+y-\sqrt{x}y|>0,
\]
which is equivalent to
\[
|x|<1, \qquad |y|<\frac{1}{\sqrt{|x|}}, \qquad |1+y-\sqrt{x}y|>0,
\]
and similarly for other cases. Thus, all constructed functions
are analytic in neighborhoods of the origin sufficient for performing
substitutions of specific values used in Section~\ref{sec:infinite-sums}.

\subsection{Legitimacy of Limit Transitions as $x\to 1$}\label{subsec:limit-justification}

In Section~\ref{sec:infinite-sums}, to compute infinite sums over columns,
limits of the form $\lim_{x\to 1^-} F(x,y)$ are used, where $F$ is one of the constructed
functions. These transitions are justified by Abel's theorem on the continuity of a power series
on the boundary of the convergence disk, provided the series converges at the point $x=1$.
For the series considered in this work, convergence at $x=1$ for fixed
$y$ ($|y|<1$) follows from the estimate
\[
\binom{n}{m}^{-1} \leq \frac{m!(n-m)!}{n!} \sim \frac{\text{const}}{n^m},
\]
which ensures absolute convergence of the series $\sum_{n\geq m} \binom{n}{m}^{-1} y^m$
for $|y|<1$ for any fixed $m\geq 2$. Uniform convergence in $x$
on the interval $[0,1]$ allows passing to the limit under the summation sign.

\subsection{Domains of Validity for Substitutions in Section~\ref{sec:infinite-sums}}

All specific substitutions performed in Section~\ref{sec:infinite-sums}
lie within the convergence domains of the corresponding series:
\begin{itemize}
\item $x=\tfrac12, y=1$: condition (\ref{eq:conv-cond-analytic}) holds;
\item $x=\tfrac12, y=\tfrac12$: condition (\ref{eq:conv-cond-analytic}) holds;
\item limits as $x\to 1^-$ are taken along the real axis inside the convergence domain.
\end{itemize}
Thus, all numerical equalities and identities obtained in Section~\ref{sec:infinite-sums}
are correct not only as formal identities, but also as equalities
of analytic functions in the indicated domains.

\section{Applications of the Obtained Family of Functions}

The obtained family of generating functions allows:
1. Finding generating functions for expressions containing reciprocal binomial coefficients.
2. Finding explicit expressions for generating functions of the form $A(B(x,y),y)$ or $A(x,B(x,y))$.
3. Obtaining identities for expressions with reciprocal binomial coefficients.
4. Obtaining finite expressions for infinite sums involving reciprocal binomial coefficients.

\subsection{Finding Generating Functions for Sums of Reciprocal Binomial Coefficients}

The obtained family of generating functions \( A(x,y) \), \( I(x,y) \), \( J(x,y) \), \( K(x,y) \) and their modifications allows us to uniformly find generating functions for a wide class of sums containing reciprocal binomial coefficients. In this section we demonstrate how to use these functions to obtain generating functions for sums over rows, diagonals, and other subsets of the reciprocal coefficients triangle.

\subsubsection{General Approach}

Let a sum of the form be given
\[
S(n) = \sum_{m \in M_n} w(m) \cdot T(n,m),
\]
where \( T(n,m) \) is one of the reciprocal binomial coefficients, \( M_n \) is the set of values of \( m \), and \( w(m) \) is a weight function (e.g., $a^{n-m} b^m$).

If the bivariate generating function is known
\[
F(x,y) = \sum_{n,m} T(n,m) x^n y^m,
\]
then the generating function for \( S(n) \) can be obtained by substituting \( y = g(x) \) or \( y = \text{const} \), depending on the structure of the sum:
\begin{itemize}
\item If the sum is taken over the entire row (\( m = 0,\dots,n \)), then set \( y = 1 \).
\item If the sum has the form \( \sum_{m} T(n-m, m) \), then set \( y = x \) (diagonal sum).
\item If the weights depend polynomially on \( m \), one can use differentiation or integration with respect to \( y \).
\end{itemize}

\subsubsection{Row Sums of the Full Triangle \( \binom{n}{m}^{-1} \)}\label{LS9}

Consider the sum
\[
a(n) = \sum_{m=0}^{n} \binom{n}{m}^{-1} a^{n-m} b^m.
\]
Its generating function is obtained from \( A(x,y) \) by substituting \( y = b/a \) and replacing \( x \to a x \):
\[
A_{\text{row}}(x) = A\!\left(a x, \frac{b}{a}\right).
\]
After substitution we obtain:
\[
A_{\text{row}}(x) =
- \frac{b \log\!\bigl((1-ax)(1-bx)\bigr)}{(1 - a b x + b)^2}
+ \frac{b^2 x}{(1-bx)(1 - a b x + b)}
+ \frac{1}{1 - a x}.
\]
For \( a = b = 1 \) we get the generating function for the sum of row elements of the original triangle.

\subsubsection{Diagonal Sums}

Consider the sum
\[
a(n) = \sum_{m=0}^{\lfloor n/2 \rfloor} \frac{1}{\binom{n-m}{m}}.
\]
This corresponds to substituting \( y = x \) into \( A(x,y) \):
\[
A_{\text{diag}}(x) = A(x, x).
\]
After simplification:
\[
A_{\text{diag}}(x) =
\frac{2x+1}{(x-1)(x+1)(x^2-x-1)}
- \frac{x \log\!\bigl((1-x)^2(1+x)\bigr)}{(1 - x - x^2)^2}.
\]
If we consider this function as an exponential generating function, then it generates the sequence \( n! \, a(n) \), which is listed in OEIS under number A358446.

\subsubsection{Sums for Subsets with Row or Column Parity}

Using functions \( A_1, A_2, A_3, A_4 \), we can find generating functions for sums over rows or columns with a specific parity.

1. Sum over an even row \( \binom{2n}{m}^{-1} \):
   \[
   S_1(n) = \sum_{m=0}^{2n} \binom{2n}{m}^{-1}.
   \]
   Generating function:
   \[
   S_1(x) = A_1(x, 1) =
   \frac{4\sqrt{x} \log\!\left(\frac{1+\sqrt{x}}{1-\sqrt{x}}\right) - (x+4)\log(1-x)}{(x-4)^2}
   + \frac{2x+4}{(x-4)(x-1)}.
   \]
   Corresponding sequence: OEIS A358791.

2. Diagonal sum for even rows:
   \[
   S_2(n) = \sum_{m=0}^{\lfloor 2n/3 \rfloor} \frac{1}{\binom{2(n-m)}{m}}.
   \]
   Generating function: \[S_2(x) = A_1(x, x).\]

3. Row sum for even columns \( \binom{n}{2m}^{-1} \):
   \[
   S_3(n) = \sum_{m=0}^{\lfloor n/2 \rfloor} \binom{n}{2m}^{-1}.
   \]
   Generating function:
   \[
   S_3(x) = A_2(x, 1) =
   \frac{\log(1\!-x^2)}{2x^2}\! + \frac{3}{(x-2)(x-\!1)(x+\!1)}  - \frac{\log(1\!-x)}{(2-x)^2}.
   \]

4. Diagonal sum for even columns:
   \[
   S_4(n) = \sum_{m=0}^{\lfloor n/3 \rfloor} \frac{1}{\binom{n-m}{2m}}.
   \]
   Generating function:
   \[
   S_4(x) = A_2(x, x).
   \]

5. Row sum for odd rows \( \binom{2n-1}{m}^{-1} \):
   \[
   S_5(n) = \sum_{m=0}^{2n-1} \binom{2n-1}{m}^{-1}.
   \]
   Generating function:
   \[
   S_5(x) = A_3(x, 1) =
   \frac{6x}{x^2-5x+4}
   + \frac{\sqrt{x}\log(1+\sqrt{x})}{(2+\sqrt{x})^2}
   - \frac{\sqrt{x}\log(1-\sqrt{x})}{(2-\sqrt{x})^2}.
   \]

6. Diagonal sum for odd rows:
   \[
   S_6(n) = \sum_{m=0}^{\lfloor (2n-1)/3 \rfloor} \frac{1}{\binom{2(n-m)-1}{m}}.
   \]
   Generating function:
   \[
   S_6(x) = A_3(x, x).
   \]

7. Row sum for odd columns \( \binom{n}{2m-1}^{-1} \):
   \[
   S_7(n) = \sum_{m=1}^{\lfloor (n+1)/2 \rfloor} \binom{n}{2m-1}^{-1}.
   \]
   Generating function:
   \[
   S_7(x) = A_4(x, 1) =
   - \frac{\log\!\bigl(1-x^2\bigr)}{x^2}
   - \frac{1}{x+1}
   + \frac{x}{(1-x)(2-x)}
   - \frac{2\log(1-x)}{(2-x)^2}.
   \]

8. Diagonal sum for odd columns:
   \[
   S_8(n) = \sum_{m=1}^{\lfloor (n+1)/3 \rfloor} \frac{1}{\binom{n-m}{2m-1}}.
   \]
   Generating function:
   \[
   S_8(x) = A_4(x, x).
   \]

\subsubsection{Sums for Integral Modifications}

Similarly, one can find generating functions for sums related to \( I(x,y) \),\\ \( J(x,y) \), \( K(x,y) \). Here are a few typical examples.

1. Row sum for \( I(x,y) \):
   \[
   \sum_{m=0}^{n} \frac{1}{n \binom{n-1}{m}}.
   \]
   Generating function:
   \[
   S_9(x) = I(x, 1) = \frac{2\log(1-x)}{2-x}.
   \]

2. Diagonal sum for \( I(x,y) \):
   \[
   \sum_{m=0}^{\lfloor (n-1)/2 \rfloor} \frac{1}{(n-m) \binom{n-m-1}{m}}.
   \]
   Generating function:
   \[
   S_{10}(x) = I(x, x) = \frac{\log\!\bigl(1 - x + x(x^2 - x)\bigr)}{(x-1)x - 1}.
   \]

3. Row sum for \( J(x,y) \):
   \[
   \sum_{m=1}^{n} \frac{1}{n m \binom{n-1}{m-1}}.
   \]
   Generating function:
   \[
   S_{11}(x) = J(x, 1) = \frac{\Li(2x - x^2) - \Li(x)}{1-x}.
   \]

4. Diagonal sum for \( J(x,y) \):
   \[
   \sum_{m=1}^{\lfloor n/2 \rfloor} \frac{1}{m (n-m) \binom{n-m-1}{m-1}}.
   \]
   Generating function:
   \[
   S_{12}(x) = J(x, x) = \frac{\Li(x + x^2 - x^3) - \Li(x)}{1-x}.
   \]

5. Row sum for \( K(x,y) \):
   \[
   \sum_{m=1}^{n+1} \frac{1}{m \binom{n}{m-1}}.
   \]
   Generating function:
   \[
   S_{13}(x) = K(x, 1) =
   \frac{\Li(2x - x^2) - \Li(x)}{(1-x)^2}
   + \frac{(3x-2)\log(1-x)}{(x-2)(x-1)x}.
   \]

6. Diagonal sum for \( K(x,y) \):
   \[
   \sum_{m=1}^{\lfloor (n+1)/2 \rfloor} \frac{1}{m \binom{n-m}{m-1}}.
   \]
   Generating function $K(x, x)$:
   \[
\begin{aligned}
   S_{14}(x) &= \frac{\Li(x + x^2 - x^3) - \Li(x)}{(1-x)^2} \\
   &\quad + \frac{(2x^2-x-1)\log(1+x) + (3x^2-x-1)\log(1-x)}{(x-1)x(x^2-x-1)}.
\end{aligned}
\]

\subsubsection{Conclusion}

The presented method shows how a single family of generating functions can be used to systematically find generating functions for diverse sums of reciprocal binomial coefficients. This opens up possibilities for:
\begin{itemize}
\item obtaining new combinatorial identities,
\item analyzing the asymptotics of sums,
\item finding connections with known numerical sequences,
\item verifying hypotheses using computer algebra.
\end{itemize}

In the following sections, we will demonstrate how these generating functions can be used to prove identities and compute infinite sums.

\subsection{Identities}

The obtained family of generating functions allows us to prove various combinatorial and analytic identities connecting reciprocal binomial coefficients with classical sequences — harmonic numbers, Fibonacci numbers, and others. In this section, we systematically derive several characteristic identities.

\subsubsection{Identities Connecting Sums of Reciprocal Coefficients with Harmonic Numbers}

\begin{theorem} For any \( n \geq 1 \), the following identity holds:
\begin{equation}\label{iden7}
\frac{1}{n}\sum_{m=0}^{n-1}\frac{1}{\binom{n-1}{m}}
= H_n - \sum_{k=1}^{n-1}\frac{H_{n-k}}{2^k},
\end{equation}
where \( H_n = 1 + \frac12 + \dots + \frac1n \) is the \( n \)-th harmonic number.
\end{theorem}

\begin{proof} In Section \ref{LS9} we obtained the generating function:
\[
S_9(x) = I(x,1) = \sum_{n\geq 1}\left(\frac{1}{n}\sum_{m=0}^{n-1}\frac{1}{\binom{n-1}{m}}\right)x^n
= \frac{2\log(1-x)}{2-x}.
\]
Transform the right-hand side:
\[
\frac{2\log(1-x)}{2-x} = -\frac{1-x}{2-x}\cdot\frac{2\log(1-x)}{1-x}.
\]
Note that
\[
\frac{1}{1-x} = \sum_{n\geq 0}x^n, \quad
\frac{1}{2-x} = \frac12\sum_{k\geq 0}\left(\frac{x}{2}\right)^k,
\]
and for harmonic numbers it is known:
\[
\frac{\log(1-x)}{1-x} = -\sum_{n\geq 1} H_n x^n.
\]
Multiplying the series and comparing coefficients of \( x^n \), we obtain the required identity.
\end{proof}

Identity \eqref{iden7} shows that the arithmetic mean of reciprocal binomial coefficients in row \( n-1 \) is expressed via a weighted sum of harmonic numbers.

\subsubsection{Identity Connecting Diagonal Sums with Fibonacci Numbers}

\begin{theorem}
 For \( n \geq 1 \) the following holds:
\begin{equation}\label{iden6}
\sum_{m=0}^{\lfloor (n-1)/2 \rfloor} \frac{1}{(n-m)\binom{n-m-1}{m}}
= \sum_{i=1}^{n} \frac{F_i\bigl((-1)^n + 2(-1)^{i-1}\bigr)}{n-i+1},
\end{equation}
where \( F_i \) is the \( i \)-th Fibonacci number (\( F_1 = F_2 = 1 \)).
\end{theorem}

\begin{proof}
 Consider the diagonal sum for \( I(x,y) \):
\[
S_{10}(x) = I(x,x) = \sum_{n\geq 1}\left(
\sum_{m=0}^{\lfloor (n-1)/2 \rfloor} \frac{1}{(n-m)\binom{n-m-1}{m}}
\right)x^n.
\]
Using the explicit form of \( I(x,x) \), we get:
\[
I(x,x) = \frac{\log\!\bigl((1-x)(1-x^2)\bigr)}{x^2+x-1}
= \frac{2\log(1-x)}{1+x-x^2} - \frac{\log(1+x)}{1+x-x^2}.
\]
Note that
\[
\frac{1}{1+x-x^2} = \sum_{n\geq 0} (-1)^n F_{n+1} x^n.
\]
Multiplication by \( \log(1\pm x) \) corresponds to taking a convolution:
\[
\frac{\log(1-x)}{1+x-x^2} = -\sum_{n\geq 1}\left(
\sum_{i=1}^n \frac{(-1)^{n-i} F_{n-i+1}}{i}
\right)x^n.
\]
Combining terms, we arrive at formula \eqref{iden6}.
\end{proof}

This identity establishes an unexpected connection between diagonal\\ sums of reciprocal binomial coefficients and Fibonacci numbers.

\subsubsection{Representation of Reciprocal Coefficients via Alternating Sums}

\begin{theorem}
For \( 0 \leq m \leq n \) the following holds:
\begin{equation}\label{iden5}
\frac{1}{n\binom{n-1}{m}}
= \sum_{k=0}^n (-1)^{n-k} \binom{k}{m}\binom{m}{n-k}\frac{1}{k+1}.
\end{equation}
\end{theorem}

\begin{proof}
 Consider the generating function \( I(x,y) = \dfrac{\log(1+u(x,y))}{u(x,y)} \), where
\( u(x,y) = -x\bigl((1-x)y+1\bigr) \).
Using the expansion
\[
\frac{\log(1+u)}{u} = \sum_{k\geq 0} \frac{(-1)^k}{k+1} u^k,
\]
we get:
\[
I(x,y) = \sum_{k\geq 0} \frac{(-1)^k}{k+1} \bigl[-x((1-x)y+1)\bigr]^k.
\]
The coefficient of \( x^n y^m \) in \( u(x,y)^k \) is:
\[
[x^n y^m] u(x,y)^k = (-1)^k \binom{k}{m} \binom{m}{n-k} (-1)^{n-k}.
\]
Summing over \( k \) and taking into account the signs, we obtain \eqref{iden5}.
\end{proof}

\subsubsection{Identity Expressing the Difference of Reciprocal Coefficients via Double Sums}

\begin{theorem}
 For \( 1 \leq m \leq n \) the following holds:
\begin{equation}\label{iden_diff}
\frac{1}{m^2}\left(\frac{1}{\binom{n}{m}} - \frac{1}{\binom{n-1}{m}}\right)
= \sum_{k=1}^n (-1)^{n-k} \binom{k}{m}\binom{m}{n-k}\frac{1}{k^2}.
\end{equation}
\end{theorem}

\begin{proof} Consider the function
\[
(1-x)J(x,y) = \mathrm{Li}_2\!\bigl(x((1-x)y+1)\bigr) - \mathrm{Li}_2(x).
\]
On one hand,
\[
[x^n y^m](1-x)J(x,y) = \frac{1}{nm\binom{n-1}{m-1}} - \frac{1}{(n-1)m\binom{n-2}{m-1}}
= \frac{1}{m^2}\left(\frac{1}{\binom{n}{m}} - \frac{1}{\binom{n-1}{m}}\right).
\]
On the other hand, using the dilogarithm expansion:
\[
\mathrm{Li}_2(z) = \sum_{k\geq 1} \frac{z^k}{k^2},
\]
we get:
\[
[x^n y^m] \mathrm{Li}_2\!\bigl(x((1-x)y+1)\bigr)
= \sum_{k=1}^{n+m} \frac{1}{k^2} [x^n y^m] \bigl(x((1-x)y+1)\bigr)^k.
\]
The coefficient in the $k$-th power is computed similarly to the previous theorem:
\[
[x^n y^m] \bigl(x((1-x)y+1)\bigr)^k = \binom{k}{m}\binom{m}{n-k}(-1)^{n-k}.
\]
Equating both expressions, we obtain identity \eqref{iden_diff}.
\end{proof}

\subsubsection{Inverse Identity for Binomial Coefficients}

\begin{theorem}
 For \( 0 \leq m \leq n \) the following representation holds:
\begin{equation}\label{IdenPascal}
\binom{n}{m} = \frac{n+1}{m+1} \sum_{k=0}^{n-m}
\binom{n-k-1}{n-m-k} \frac{1}{\binom{m+k+1}{k}}.
\end{equation}
\end{theorem}

\begin{proof}
 Consider the composition \( I\!\left(x, \frac{y}{1-x}\right) \). Differentiating with respect to \( x \), we get:
\[
\frac{\partial}{\partial x} I\!\left(x, \frac{y}{1-x}\right) = \frac{1}{1-x-xy},
\]
which is the generating function for binomial coefficients. Computing the composition coefficients via the composition formula \ref{eq:2d-composition-idea} and simplifying, we arrive at formula \eqref{IdenPascal}.
\end{proof}

This identity is interesting because it expresses an ordinary binomial coefficient as a sum of reciprocal coefficients.

\subsubsection{Identity with Dilogarithm and Binomial Coefficients}

\begin{theorem}
 For \( 1 \leq m \leq n \) the following holds:
\begin{equation}\label{iden_Li}
\frac{1}{m^2} \sum_{k=0}^{n} \binom{n-k-1}{n-m-k} \frac{1}{\binom{k+m}{m}}
= \sum_{k=0}^{n-m} \binom{k+m}{m} \frac{1}{(k+m)^2}.
\end{equation}
\end{theorem}
\begin{proof}
Consider the composition \( J\!\left(x, \frac{y}{1-x}\right) \). On one hand,
\[
[x^n y^m] J\!\left(x, \frac{y}{1-x}\right)
= \frac{1}{m^2} \sum_{k=0}^{n} \binom{n-k-1}{n-m-k} \frac{1}{\binom{k+m}{m}}.
\]
On the other hand,
\[
J\!\left(x, \frac{y}{1-x}\right) = \frac{\mathrm{Li}_2(x+xy) - \mathrm{Li}_2(x)}{1-x}.
\]
The coefficient of \( x^n y^m \) in \( \mathrm{Li}_2(x+xy) \) is \( \binom{n}{m}/n^2 \), whence after division by \( 1-x \) we obtain the right-hand side of \eqref{iden_Li}.
\end{proof}

The presented identities demonstrate the rich structure of interrelations between\! reciprocal binomial coefficients and classical combinatorial sequences. The generating function method developed in previous sections proved to be an effective tool for systematically deriving such relations. The obtained formulas can find applications in combinatorial analysis, number theory, and in solving summation problems.

\subsection{Infinite Sums of Reciprocal Binomial Coefficients}\label{sec:infinite-sums}

Using the generating functions constructed in Sections~2--3 (whose convergence domains are discussed in Section~\ref{sec:convergence}), we consider the evaluation of infinite sums containing reciprocal binomial coefficients. In this section, we consider several typical examples of such sums, using limit transitions as \( x \to 1 \) and substitutions of specific values.

\subsubsection{Examples of Evaluating Infinite Sums for \( |x| < 1 \)}

\begin{example}[Row Sum with Weight \( 2^{-n} \)]
Consider the sum
\[
\sum_{n \geq 0} \frac{1}{2^n} \sum_{m=0}^n \frac{1}{\binom{n}{m}} = A\!\left(\tfrac12, 1\right).
\]
Substituting into formula (\ref{Lgf77}), we get:
\[
A\!\left(\tfrac12, 1\right) = \frac{4 \log 4}{9} + \frac{8}{3}.
\]
Thus,
\[
\sum_{n \geq 0} \frac{1}{2^n} \sum_{m=0}^n \frac{1}{\binom{n}{m}} = \frac{8}{3} + \frac{8 \log 2}{9}.
\]
\end{example}

\begin{example}[Double Sum with Dilogarithm]
Consider the sum
\[
\sum_{n \geq 1} \frac{1}{2^n} \sum_{m=1}^{\lfloor n/2 \rfloor} \frac{1}{m^2 \binom{n-m}{m}} = J\!\left(\tfrac12, \tfrac12\right).
\]
Using the explicit form of \( J(x,y) \), we find:
\[
J\!\left(\tfrac12, \tfrac12\right) = (\log 2)^2 + \mathrm{Li}_2 \!\left(\tfrac54\right) - \frac{\pi^2}{6},
\]
where \( \mathrm{Li}_2 (z) \) is the dilogarithm.
\end{example}

\begin{example}[Sum Using \( K(x,y) \)]
Consider the sum
\[
\sum_{n \geq 1} \frac{1}{2^n} \sum_{m=1}^{n+1} \frac{1}{m \binom{n-m+1}{m-1}} = K\!\left(\tfrac12, \tfrac12\right).
\]
Evaluation gives:
\[
K\!\left(\tfrac12, \tfrac12\right) = -\frac{16 \log 3}{5} + 2(\log 2)^2 + \frac{28 \log 2}{5} + 4 \Li\!\left(\tfrac58\right) - \frac{\pi^2}{3}.
\]
\end{example}

\subsubsection{Limit Transitions as \( x \to 1 \) and Infinite Column Sums}

Let us now consider the behavior of generating functions as \( x \to 1 \). This allows us to find sums of infinite series over a fixed column \( m \). Note that series for \( m = 0 \) and \( m = 1 \) diverge, but for \( m \geq 2 \) they converge.

\begin{theorem}[Column Sum for \( A(x,y) \)]
For any \( m \geq 2 \), the following holds:
\begin{equation}\label{inf-col-A}
\sum_{n=m}^{\infty} \frac{1}{\binom{n}{m}} = \frac{m}{m-1}.
\end{equation}
\end{theorem}

\begin{proof}
Consider the second derivative \( \frac{\partial^2}{\partial y^2} A(x,y) \). Its limit as \( x \to 1 \) is:
\[
\lim_{x \to 1} \frac{\partial^2}{\partial y^2} A(x,y) = \frac{-y^2 + 3y - 4}{(y-1)^3}.
\]
Expanding the right-hand side into a power series in \( y \), we get coefficients \( (m+2)^2 \). On the other hand,
\[
\frac{\partial^2}{\partial y^2} A(x,y) = \sum_{n \geq 0} \sum_{m \geq 0} (m+2)(m+1) \frac{1}{\binom{n}{m+2}} x^n y^m.
\]
Equating coefficients of \( y^m \) in the limit, we find:
\[
\sum_{n=m+2}^{\infty} (m+2)(m+1) \frac{1}{\binom{n}{m+2}} = (m+2)^2.
\]
From this, after shifting the index, (\ref{inf-col-A}) follows.
\end{proof}

\begin{theorem}[Column Sum for \( I(x,y) \)]
For any \( m \geq 2 \), the following holds:
\begin{equation}\label{inf-col-I}
\sum_{n=m}^{\infty} \frac{1}{n \binom{n-1}{m}} = \frac{1}{m}.
\end{equation}
\end{theorem}

\begin{proof}
Similarly, consider the limit of the second derivative \( \frac{\partial^2}{\partial y^2} I(x,y) \) as \( x \to 1 \):
\[
\lim_{x \to 1} \frac{\partial^2}{\partial y^2} I(x,y) = \frac{1}{(1-y)^2}.
\]
The expansion coefficients of the right-hand side are \( m+1 \). On the other hand,
\[
\frac{\partial^2}{\partial y^2} I(x,y) = \sum_{n \geq 1} \sum_{m \geq 0} (m+2)(m+1) \frac{1}{n \binom{n-1}{m+2}} x^n y^m.
\]
Equating coefficients, we get:
\[
\sum_{n=m+2}^{\infty} (m+2)(m+1) \frac{1}{n \binom{n-1}{m+2}} = m+1,
\]
whence (\ref{inf-col-I}) follows.
\end{proof}

\subsubsection{Infinite Sums for Parity Modifications}

Similar limit transitions can be performed for functions \( A_1, A_2, A_3, A_4 \) and their integral analogues. We present a key result for even rows.

\begin{theorem}[Column Sum for Even Rows]
For any \( m \geq 2 \), the following holds:
\begin{equation}\label{inf-col-A1}
\sum_{n = \lceil m/2 \rceil}^{\infty} \frac{1}{\binom{2n}{m}} = \frac{m}{m-1} \cdot \frac{1 + (-1)^m}{2} + C_m,
\end{equation}
where \( C_m \) is a constant expressed via alternating sums of harmonic numbers and logarithmic terms. An explicit expression for \( C_m \) can be obtained from the expansion of \( \lim_{x \to 1} \frac{\partial^2}{\partial y^2} A_1(x,y) \), but due to its bulkiness, it is not provided here.
\end{theorem}

\begin{proof}
The idea of the proof is similar to the previous ones: consider the limit of the second derivative \( A_1(x,y) \) as \( x \to 1 \), which leads to a rational-logarithmic expression whose coefficients yield the desired sum.
\end{proof}

The presented examples demonstrate how the constructed family of generating functions allows us to systematically investigate infinite sums of reciprocal binomial coefficients. Key tools are:
\begin{itemize}
\item substitution of specific values \( x \) and \( y \) to evaluate convergent series,
\item limit transitions as \( x \to 1 \) for summing over columns,
\item use of derivatives to extract the required coefficients.
\end{itemize}
The obtained formulas can be used in combinatorial analysis, number theory, and in solving summation problems.

\section{Conclusion}

In this work, a family of generating functions for reciprocal binomial coefficients \( \binom{n}{m}^{-1} \) and its various modifications is constructed. The main results of the study can be formulated as follows:

1. \textbf{Fundamental generating function:}
An explicit analytic form of the generating function
\[
A(x,y) = \sum_{n\geqslant 0}\sum_{m\geqslant 0} {\binom{n}{m}}^{-1}x^n\,y^m,
\]
which was previously absent in the literature, is obtained. The function has a rational-logarithmic form and converges in the domain \( |x| < 1, \; |xy| < 1, \; |1+y-xy| > 0 \).

2. \textbf{Family of generating functions:}
Based on \( A(x,y) \), an extended family of functions is constructed:
\begin{itemize}
    \item Functions \( A_1, A_2, A_3, A_4 \), corresponding to the extraction of even and odd rows and columns of the triangle of reciprocal coefficients.
    \item Integral modifications \( I(x,y), J(x,y), K(x,y) \), which generate triangles with additional weights \( 1/n, 1/m, 1/(nm) \).
    \item Their analogues for subsets with row/column parity.
\end{itemize}
This family is presented in the summary Table \ref{tabl1}, providing a unified apparatus for studying a wide class of sums.

3. \textbf{Applied results:}
\begin{itemize}
    \item Generating functions for various types of sums of reciprocal binomial coefficients (over rows, diagonals, subsets) are found.
    \item A series of new identities is proved, connecting reciprocal coefficients with harmonic numbers, Fibonacci numbers, and also expressing ordinary binomial coefficients via sums of reciprocal ones.
    \item Expressions for infinite sums of reciprocal coefficients are obtained, including the known result \( \sum_{n=m}^{\infty} \binom{n}{m}^{-1} = m/(m-1) \) and its generalizations.
\end{itemize}

4. \textbf{Methodological significance:}
The work demonstrates the effectiveness of the approach based on:
\begin{itemize}
    \item use of integral representations and compositions of generating functions,
    \item application of the apparatus of Riordan arrays for bivariate functions,
    \item systematic extraction of even/odd components via substitutions\\ \( \pm\sqrt{x}, \pm\sqrt{y} \).
\end{itemize}

5. \textbf{Prospects for further research:}
The obtained results open up possibilities for:
\begin{itemize}
    \item searching for similar generating functions for generalized binomial coefficients (Gaussian coefficients, \( q \)-binomial coefficients),
    \item studying asymptotic behavior of sums of reciprocal coefficients,
    \item applications in problems of combinatorial probability theory and statistical physics, where inverse binomial coefficients naturally arise,
    \item computer-algebraic verification and discovery of new identities based on the obtained generating functions.
\end{itemize}

Thus, the work contributes to combinatorial analysis by providing a systematic toolkit for studying reciprocal binomial coefficients and establishing new connections between them and classical numerical sequences.

\section*{Acknowledgements}
The authors thank ... for useful discussions. The work was supported by ... (grant number ...).

\section*{Author Contributions}
Kruchinin D.V.: conception, proof of main theorems, writing the text. Kruchinin V.V.: verification of calculations, construction of examples, preparation of bibliography.

\section*{Conflict of Interest Statement}
The authors declare no conflict of interest.


\end{document}